\documentclass[11pt,onecolumn,oneside]{article}
\usepackage[a4paper, hmargin={2.8cm, 2.8cm}, vmargin={2.5cm, 2.5cm}]{geometry}

\usepackage[utf8]{inputenc}      		 
\usepackage[T1]{fontenc}                 
\usepackage{graphicx, color}             
\usepackage{textcomp}                    
\usepackage{amssymb}
\usepackage{amsmath}
\usepackage{amsthm}						 
\usepackage{mathrsfs}
\usepackage{caption, subcaption}
\usepackage{wrapfig}
\usepackage{gensymb}
\usepackage{mathtools}
\usepackage{bm} 
\usepackage{subfiles}
\usepackage[normalem]{ulem}              

\DeclareFontFamily{U}{mathx}{\hyphenchar\font45}
\DeclareFontShape{U}{mathx}{m}{n}{
      <5> <6> <7> <8> <9> <10>
      <10.95> <12> <14.4> <17.28> <20.74> <24.88>
      mathx10
      }{}
\DeclareSymbolFont{mathx}{U}{mathx}{m}{n}
\DeclareFontSubstitution{U}{mathx}{m}{n}
\DeclareMathAccent{\widecheck}{0}{mathx}{"71}
\DeclareMathAccent{\wideparen}{0}{mathx}{"75}

\usepackage{tikz}
\usetikzlibrary{cd}
\usepackage{verbatim}

\usepackage{enumerate}	      	
\usepackage{hyperref}			
\usepackage{xcolor}				
\hypersetup{					
    colorlinks,
    linkcolor={red!50!black},
    citecolor={green!30!black},
    urlcolor={blue!80!black},
}
\usepackage{xkvltxp} 
\usepackage[draft,nomargin,inline,index]{fixme} 
\fxsetup{theme=color}

\newcommand{\N}{\ensuremath{\mathbb{N}}}

\newcommand{\R}{\ensuremath{\mathbb{R}}}
\newcommand{\C}{\ensuremath{\mathbb{C}}}

\newcommand{\T}{\ensuremath{\mathbb{T}}}

\newcommand{\Strip}{\ensuremath{\mathbb{S}}}

\DeclareMathOperator{\supp}{supp}  
\newcommand{\restrict}[1]{\raisebox{-0.15ex}{$|$}_{#1}}
\newcommand{\integral}[3]{\int_{#2} #1\,\mathrm{d}#3}

\newcommand{\rintegral}[4]{\int_{#2}^{#3} #1\,\mathrm{d}#4}

\DeclareMathOperator{\real}{Re}
\DeclareMathOperator{\imag}{Im}

\DeclareMathOperator{\aut}{Aut} 
\DeclareMathOperator{\isom}{Isom} 

\DeclareMathOperator{\Rep}{Rep}

\newcommand{\act}[1]{\overset{#1}{\curvearrowright}}
\newcommand{\overdot}[1]{\overset{\boldsymbol{\cdot}}{#1}}

\newcommand{\set}[2]{\left\{\,#1\;\middle|\; #2\,\right\}}

\makeatletter
\def\moverlay{\mathpalette\mov@rlay}
\def\mov@rlay#1#2{\leavevmode\vtop{%
   \baselineskip\z@skip \lineskiplimit-\maxdimen
   \ialign{\hfil$\m@th#1##$\hfil\cr#2\crcr}}}
\newcommand{\charfusion}[3][\mathord]{
    #1{\ifx#1\mathop\vphantom{#2}\fi
        \mathpalette\mov@rlay{#2\cr#3}
      }
    \ifx#1\mathop\expandafter\displaylimits\fi}
\makeatother
\newcommand{\bigcupdot}{\charfusion[\mathop]{\bigcup}{\cdot}}

\newcommand{\net}[3]{\left(#1_{#2}\right)_{#2\in#3}}

\newcommand{\inner}[2]{\left\langle #1\, ,#2 \right\rangle}

\newcommand{\innerdot}{\left\langle \,\cdot\, ,\,\cdot\, \right\rangle}

\newcommand{\norm}[2]{\ensuremath{\left|\!\left|#1\right|\!\right|_{#2}}}

\newcommand{\normdot}[1]{\ensuremath{\left|\!\left|\,\cdot\,\right|\!\right|_{#1}}}

\newcommand{\abs}[1]{\ensuremath{\left|#1\right|}}

\usepackage{etoolbox}
\newcommand{\addQEDstyle}[2]{\AtBeginEnvironment{#1}{\pushQED{\qed}\renewcommand{\qedsymbol}{#2}}\AtEndEnvironment{#1}{\popQED}}

\theoremstyle{plain}
\newtheorem{thm}{Theorem}
\numberwithin{thm}{section}
\newtheorem{lem}[thm]{Lemma}
\newtheorem{cor}[thm]{Corollary}
\newtheorem{prop}[thm]{Proposition}

\theoremstyle{definition}

\addQEDstyle{ex}{$\circ$} 

\newtheoremstyle{warning}
{\topsep}
{\topsep}
{}
{}
{\bfseries}
{!}
{.5em}
{}
\theoremstyle{warning}
\newtheorem*{warning}{Warning}

\newtheoremstyle{question}
{\topsep}
{\topsep}
{}
{}
{\bfseries}
{.}
{.5em}
{}
\theoremstyle{question}

\theoremstyle{remark}
\newtheorem{rk}[thm]{Remark}

\theoremstyle{plain}
\newtheorem{introthm}{Theorem}

\theoremstyle{definition}

\title{Symmetrized pseudofunction algebras from $L^p$-representations and amenability of locally compact groups}
\author{Emilie Mai Elki\ae{}r}
\date{\today}

\begin{document}

\maketitle

\begin{abstract}
We show via an application of techniques from complex interpolation theory how the $L^p$-pseudofunction algebras of a locally compact group $G$ can be understood as sitting between $L^1(G)$ and $C^*(G)$. Motivated by this, we collect and review various characterizations of group amenability connected to the $p$-pseudofunction algebra of Herz and generalize these to the symmetrized setting. Along the way, we describe the Banach space dual of the symmetrized pseudofuntion algebras on $G$ associated with representations on reflexive Banach spaces.
\end{abstract}

\section{Introduction}
Let $G$ be a locally compact group. For $1<p<\infty$, the \emph{$p$-pseudofunction algebra} of $G$, which we denote by $F_{\lambda_p}(G)$, is the completion of $L^1(G)$ with respect to the norm associated with the left-regular representation of $G$ on $L^p(G)$. This Banach algebra goes back to the work of Herz from the 1970's (see Section 8 in \cite{Herz1973HarmonicSynthesis}) where it is denoted by $PF_p(G)$, and it has been studied intensely in the context of abstract harmonic analysis (see, e.g., \cite{Herz1976asymmetry}, \cite{CowlingFournier1976Inclusions}, \cite{DerighettiConvolution} and \cite{DawsSpronk2013TheAP}). More recently, it has appeared in work by several authors playing the role as an $L^p$-analog of the reduced group $C^*$-algebra. For example, the simplicity of $F_{\lambda_p}(G)$ is studied by Hejazian and Pooya in \cite{HejazianPooya2014Simple} and by Phillips in \cite{Phillips2019ReducedGpBanachAlg}. Further, Liao and Yu studies the K-theory of $F_{\lambda_p}(G)$ in \cite{LiaoYu2017Ktheory} and Gardella and Thiel shows in \cite{GardellaThiel2018Isomorphisms} the strong rigidity result that a locally compact group $G$ can be recovered from $F_{\lambda_p}(G)$ when $1<p<\infty$ and $p\neq2$. The analogy with the reduced group $C^*$-algebra is underlined by Phillips who refers to $F_{\lambda_p}(G)$ as the \emph{reduced group $L^p$-operator algebra of $G$}. Analogs of the universal group $C^*$-algebra have also appeared in the literature in the form of various pseudofunction algebras, where by \emph{pseudofunction algebra} we mean a completion of $L^1(G)$ with respect to a norm coming from some class of isometric representations of $G$. For example, in the work of Gardella and Thiel in \cite{GardellaThiel2014GpAlgActingOnLp}, the role of $C^*(G)$ is played by the pseudofunction algebra $F_{L^p}(G)$ associated with the class of isometric representations of $G$ on $L^p$-spaces. The $L^p$-pseudofunction algebra $F_{L^p}(G)$ is also suggested as a natural $p$-analog of $C^*(G)$ by Drutu and Nowak in \cite{DrutuNowak2015KazhdanProj}. Another suggestion for a $p$-analog of $C^*(G)$ is the $QSL^p$-pseudofunction algebra $F_{QSL^p}(G)$, which appears in \cite{Runde2004repQSLP} where Runde studies its dual as a $p$-analog of the Fourier-Stieltjes algebra.

In this paper, we study symmetrized versions of pseudofunctions algebras on $G$ associated with representations on $L^p$- and on $QSL^p$-spaces. We are, in particular, interested in the symmetrized $p$-pseudofunction algebras $F^*_{\lambda_p}(G)$ and the symmetrized $L^p$-pseudofunction algebras $F^*_{L^p}(G)$. The symmetrized $p$-pseudofunction algebras were introduced by Liao and Yu in \cite{LiaoYu2017Ktheory} in connection with the Baum-Connes conjecture. Later, their simplicity has been studied by Phillips in \cite{Phillips2019ReducedGpBanachAlg}, and they have appeared in the work of Samei and Wiersma in \cite{Samei2018QuasiHermitian} and \cite{SameiWiersma2018Exotic} where they were studied in connection with quasi-Hermitian groups and exotic group $C^*$-algebras, respectively. Unlike their non-symmetrized relatives, the symmetrized pseudofunction algebras are always Banach $^*$-algebras with the involution coming from $L^1(G)$. The starting point for this paper is the commutative diagram below consisting of canonical contractions with dense range:
\begin{center}
\begin{tikzcd}[row sep=normal, column sep = normal]
~ & ~ & L^1(G) \arrow{dl}{} \arrow{dr}{} & ~ & ~ \\ 
~ & F_{L^p}^*(G) \arrow{dl}{} \arrow{rr}{} & ~ & F_{\lambda_p}^*(G) \arrow{dr}{} & ~\\
C^*(G) \arrow{rrrr}{} & ~ & ~ & ~ & C_r^*(G)
\end{tikzcd}
\end{center}
With this diagram in mind, we think of $F^*_{\lambda_p}(G)$ and $F^*_{L^p}(G)$ as interpolations between $L^1(G)$ and the group $C^*$-algebras. More precisely, for a pair of Hölder exponents $1\leq p<q\leq2$, it is shown in Proposition 4.5 in \cite{Samei2018QuasiHermitian} that the identity on $L^1(G)$ extends to a contraction $F_{\lambda_p}^*(G)\rightarrow F_{\lambda_q}^*(G)$. This is known to fail in general in the non-symmetrized setting (see Remark 3.19 in \cite{GardellaThiel2014GpAlgActingOnLp}). In Theorem \ref{introthm:FstarLp->FstarLq}, we show the analogous statement for the symmetrized $L^p$-pseudofunction algebras via an application of Stein's Interpolation Theorem. For the non-symmetrized $L^p$-pseudofunction algebras, this is shown in Theorem 2.30 in \cite{GardellaThiel2014GpAlgActingOnLp}. In the non-symmetrized setting, however, the proof is more cumbersome as interpolation techniques are not available.
\begin{introthm}\label{introthm:FstarLp->FstarLq}
Let $G$ be a locally compact group and let $1\leq p<q\leq2$, The identity map on $L^1(G)$ extends to a contractive $^*$-homomorphism $F_{L^p}^*(G)\rightarrow F_{L^q}^*(G)$ with dense range. 
\end{introthm}

The notion of amenability goes back to work of von Neumann in \cite{vonNeumannTheorieDesMasses} and is originally defined in the context of measure theory. Since then, it has proven itself a fundamental concept with equivalent characterizations coming from many different corners of mathematics. We refer to \cite{Runde2002Amenability}, \cite{Pier1984AmenableLCGP}, Section 2.6 in \cite{BrownOsawa} or Appendix G in \cite{BekkaDeLaHarpeValette} for introductions to the topic and for an overview of its many connections. In $C^*$-algebraic terms, it is the property that the universal and reduced group $C^*$-algebras coincide canonically, or, equivalently, that the trivial representation extends to a $^*$-representation of $C^*_r(G)$. With the diagram presented in the previous paragraph in mind, the question of a symmetrized $L^p$-generalization naturally arizes: Is $G$ amenable if and only if $F^*_{L^p}(G)$ and $F^*_{\lambda_p}(G)$ coincide canonically? We show in Theorem \ref{introthm:amenability_and_Fstarp} that the answer to this question is affirmative. This extends Proposition 3.1 of \cite{SameiWiersma2018Exotic}, which states that $G$ is amenable if and only if the trivial representation extends to a $^*$-representation of $F^*_{\lambda_p}(G)$. Further, we build upon the work of Cowling in \cite{Cowling1979Littlewool-Paley} and Runde in \cite{Runde2004repQSLP} and give a characterization of amenability in terms of the Banach space dual of the symmetrized $p$-pseudofunction algebra, $F^*_{\lambda_p}(G)'$, and the $p$-Fourier-Stieltjes algebra, $B_p(G)$, introduced by Runde in \cite{Runde2004repQSLP}. This result is also included in Theorem \ref{introthm:amenability_and_Fstarp}.
\begin{introthm}\label{introthm:amenability_and_Fstarp}
Let $G$ be a locally compact group and let $1<p,p'<\infty$ be Hölder conjugates. The following are equivalent:
\begin{enumerate}[(i)]
\item $G$ is amenable,
\item $F^*_{\lambda_p}(G)'$ is canonically isometrically isomorphic to the sum space $B_p(G)+B_{p'}(G)$,
\item The canonical map $F^*_{L^p}(G)\rightarrow F^*_{\lambda_p}(G)$ is an isometric isomorphism,
\item The canonical map $F^*_{L^p}(G)\rightarrow F^*_{\lambda_p}(G)$ is an isomorphism,
\item The trivial representation $1_G$ extends to a $^*$-representation of $F^*_{\lambda_p}(G)$.
\end{enumerate}
\end{introthm}
The equivalence of amenability and the properties (iii) and (iv) of Theorem \ref{introthm:amenability_and_Fstarp} should be seen as parallel to Theorem 3.7 in \cite{GardellaThiel2014GpAlgActingOnLp} where the analogous equivalences are established in the non-symmetrized setting.

To establish the equivalence of amenability and property (ii) of Theorem \ref{introthm:amenability_and_Fstarp}, we characterize the Banach space dual of the symmetrized pseudofunction algebra $F^*_{\pi}(G)$ belonging to a general isometric representation $\pi$ of $G$ on a reflexive Banach space, e.g., an $L^p$- or $QSL^p$-space with $1<p<\infty$. This may be of independent interest.
\begin{introthm}\label{introthm:Bpi+Bpiprime_is_dual_of_Fstarpi}
Let $(\pi,E)$ be an isometric representation of the locally compact group $G$ on a reflexive Banach space $E$. The canonical identification of $L^1(G)'$ and $L^{\infty}(G)$ restricts to an isometric isomorphism between $F_{\pi}^*(G)'$ and the sum space $B_{\pi}(G)+B_{\pi'}(G)$.
\end{introthm}

Here, the norm on the sum space is given, for $\varphi\in B_{\pi}(G)+B_{\pi'}(G)$, by
\[
\norm{\varphi}{B_{\pi}+B_{\pi'}}=\inf\set{\norm{\varphi_0}{B_{\pi}}+\norm{\varphi_1}{B_{\pi'}}}{\varphi=\varphi_0+\varphi_1,\varphi_0\in B_{\pi}(G),\varphi_1\in B_{\pi'}(G)}.
\]

This paper is organized as follows: In Section \ref{sec:pre}, we recall the $L^p$-representation theory for a locally compact group when $p\neq2$. Further, we recall the construction of a symmetrized pseudofunction algebra, and we introduce the tools from complex interpolation theory which we shall need in the paper. In Section \ref{sec:interpolations}, we prove Theorem \ref{introthm:FstarLp->FstarLq}. In Section \ref{sec:dual_of_Fstar}, we study the Banach space dual of a pseudofunction algebra and prove Theorem \ref{introthm:Bpi+Bpiprime_is_dual_of_Fstarpi}. Finally, in Section \ref{sec:amenability}, we discuss applications to amenability and prove Theorem \ref{introthm:amenability_and_Fstarp}.

\paragraph{Acknowledgements}
The author thanks Nadia Larsen for many discussions, Mikael de la Salle for sharing his insights into complex interpolation theory and Matthew Wiersma for sharing how to extend the proof of Corollary \ref{cor:KestenHulanickiReiter} to the non-$\sigma$-finite case. The author is grateful to Eusebio Gardella and Hannes Thiel for many useful comments on an earlier draft of this paper that helped improve its presentation.



\section{Preliminaries}\label{sec:pre}
\paragraph{Group actions on $L^p$-spaces}
Let $(\Omega,\mu)$ be a $\sigma$-finite measure space and let $1<p<\infty$, $p\neq2$ be fixed. The group $\isom(L^p(\Omega,\mu))$ of surjective isometries of $L^p(\Omega,\mu)$ is described completely by the Banach-Lamperti Theorem, which we recall in Theorem \ref{thm:Banach-Lamperti} below. There are two basic types of isometries on $L^p(\Omega,\mu)$:
\begin{enumerate}[1.]
\item We denote by $L^0(\Omega,\mu;\T)$ the collection of measurable functions on $\Omega$ with values in the unit circle $\T$ with two functions identified if they differ only on a null set. For each $c\in L^0(\Omega,\mu;\T)$, the associated \emph{multiplier} $m_c$ is the surjective isometry on $L^p(\Omega,\mu)$ given by, for $\xi\in L^p(\Omega,\mu)$,
\[m_c(\xi)=c\cdot \xi.\]
With multiplication defined pointwise, $L^0(\Omega,\mu;\T)$ is a group, and we obtain an injective group homomorphism $m:L^0(\Omega,\mu;\T)\rightarrow\isom(L^p(\Omega,\mu))$ by setting $m(c)=m_c$.
\item We denote by $\aut(\Omega,[\mu])$ the group of all bi-measurable transformations $\sigma$ of $\Omega$ that leave $\mu$ quasi-invariant, i.e., the push forward measure $\sigma_*\mu$ of $\mu$ under $\sigma$ has the same null sets as $\mu$. This assumption ensures the existence of the Radon-Nikodym derivative $\tfrac{\mathrm{d}\sigma_*\mu}{\mathrm{d}\mu}$, which is a real-valued and non-negative function on $\Omega$. Define for each $\xi\in L^p(\Omega,\mu)$,
\begin{align*}
u_{\sigma}\xi=\left(\frac{\mathrm{d}\sigma_*\mu}{\mathrm{d}\mu}\right)^{1/p}\xi\circ\sigma^{-1}
\end{align*}
Then $u_{\sigma}$ is a surjective isometry on $L^p(X,\mu)$. We obtain an injective map $u:\aut(\Omega,[\mu])\rightarrow \isom(L^p(\Omega,\mu))$ by setting $u(\sigma)=u_{\sigma}$. It follows from the uniqueness part of the Radon-Nikodym Theorem that this is a group homomorphism.
\end{enumerate}

For each pair $c\in L^0(\Omega,\mu;\T)$ and $\sigma\in\aut(\Omega,[\mu])$, a straight forward computation verifies that they satisfy the covariance relation $u_{\sigma}m_cu_{\sigma}^{-1} = m_{c\circ\sigma^{-1}}$. Thus, we have an injective group homomorphism 
\begin{center}
\begin{tikzcd}[row sep=tiny, column sep = normal]
L^0(\Omega,\mu;\T)\rtimes\aut(\Omega,[\mu]) \arrow{r}{} & \isom(L^p(\Omega,\mu))\\
\qquad (c,\sigma)\qquad  \arrow[mapsto]{r}{} & \qquad m_cu_{\sigma}\qquad \\
\end{tikzcd}
\end{center}
The content of the Banach-Lamperti theorem is that this map is surjective when $p\neq 2$. This was proven by Lamperti in \cite{Lamperti1958Isom} and prior to that claimed without proof in the monograph \cite{Banach1932} by Banach in the special case of the interval equipped with the Lebesgue measure. A modern account in the general setting of Boolean algebras can be found in \cite{Gardella2019ModernLook}.
\begin{thm}[Banach-Lamperti]\label{thm:Banach-Lamperti}
Let $(\Omega,\mu)$ be a $\sigma$-finite measure space, let $1\leq p<\infty$, $p\neq2$, and let $T$ be a surjective isometry on $L^p(\Omega,\mu)$. There exist unique $c\in L^0(\Omega,\mu;\T)$ and $\sigma\in\aut(\Omega,[\mu])$ such that $T=m_cu_{\sigma}$.
\end{thm}

Let $G$ be a locally compact group and let $(\Omega,\mu)$ be a $\sigma$-finite measure space. An isometric representation of $G$ on $L^p(\Omega,\mu)$ is a strongly continuous group homomorphism $G\rightarrow\isom(L^p(\Omega,\mu))$. As a corollary to the Banach-Lamperti Theorem, we get a complete description of the isometric representations of $G$ on $L^p(\Omega,\mu)$ when $p\neq2$. Before stating this in Corollary \ref{cor:isom_reps_on_Lp}, we shall need to recall the definition of a measure class preserving action and a $1$-cocycle for such an action.

A measure class preserving action of $G$ on a $\sigma$-finite measure space $(\Omega,\mu)$ is a group homomorphism $\sigma:G\rightarrow\aut(\Omega,[\mu])$. We write $G\act{\sigma}(\Omega,\nu)$ for the action given by $\sigma$. We shall often omit $\sigma$ from the notation and write $t.\omega$ rather than $\sigma_t(\omega)$, for $t\in G$ and $\omega\in\Omega$. A ($\T$-valued) $1$-cocycle for the action $G\act{\sigma}(\Omega,\nu)$ is a map $c:G\times\Omega\rightarrow\T$ which satisfies the $1$-cocycle relation $c_{st}=c_s\cdot(c_t\circ\sigma_s^{-1})$ $\mu$-a.e., for every pair $s,t\in G$, and such that $c_t$ is a measurable map, for every $t\in G$. The set of all $1$-cocycles for $\sigma$ is denoted by $Z^1(\sigma;\T)$. Given a measure class preserving action $G\act{\sigma}(\Omega,\nu)$ and a $1$-cocycle $c\in Z^1(\sigma;\T)$, we construct an isometric representation of $G$ on $L^p(\Omega,\mu)$ as follows: For $t\in G$, set $\pi_{p,\sigma,c}(t)=m_{c_t}u_{\sigma_t}$. That is, for $\xi\in L^p(\Omega,\mu)$ and $\omega\in\Omega$,
\begin{align}\label{eqn:def:pi_p,sigma,c}
\pi_{p,\sigma,c}(t)\xi(\omega)=c_t(\omega)\left(\frac{\mathrm{d}s.\mu}{\mathrm{d}\mu}\right)^{1/p}(\omega)\xi(s^{-1}.\omega).
\end{align}

\begin{cor}\label{cor:isom_reps_on_Lp}
Let $G\act{\sigma}(\Omega,\mu)$ be a measure class preserving action of a locally compact group on a $\sigma$-finite measure space, let $c\in Z^1(\sigma;\T)$ and let $1\leq p<\infty$. Then $\pi_{p,\sigma,c}$ defined in equation (\ref{eqn:def:pi_p,sigma,c}) is an isometric representation of $G$ on $L^p(\Omega,\mu)$. Moreover, if $p\neq2$, all isometric representations on $L^p(\Omega,\mu)$ have this form.
\end{cor}

\paragraph{Symmetrized pseudofunction algebras}
We recall in the following the construction of a symmetrized pseudofunction algebra. We refer to \cite{DrutuNowak2015KazhdanProj}, \cite{ElkiaerPooya2023} and \cite{GardellaThiel2014GpAlgActingOnLp} for a more thorough treatment of general pseudofunction algebras. 

Let $G$ be a locally compact group. Given an isometric representation $\pi$ of $G$ on a Banach space $E$, its integrated form is the contractive, non-degenerate Banach algebra representation of $L^1(G)$ on $E$ given, for $f\in L^1(G)$, by
\[
\pi(f)=\integral{f(s)\pi(s)}{G}{\mu_G(s)},
\]
where $\mu_G$ denotes the Haar-measure on $G$. It is folklore that integration gives a $1$-to-$1$ correspondence between the isometric representations of $G$ and the contractive, non-degenerate representations of $L^1(G)$. We denote by $E'$ the Banach space dual of $E$. The isometric representation $(\pi,E)$ gives, in a natural way, rise to an isometric representation on $E'$ as follows: For $t\in G$, $\xi\in E$ and $\eta\in E'$, set
\[
[\pi'(t)\eta](\xi)=\eta(\pi(t^{-1})\xi).
\]
We refer to $(\pi',E')$ as the \emph{dual representation} of $(\pi,E)$. We assume in the following that $E$ is reflexive so that $\pi''$ can be identified with $\pi$. We associate to $\pi$ a seminorm on $L^1(G)$ as follows: For $f\in L^1(G)$, set
\[
\norm{f}{F^*_{\pi}}=\max\{\norm{\pi(f)}{},\norm{\pi'(f)}{}\}.
\]
This defines a norm on the quotient of $L^1(G)$ with $\ker\pi\cap\ker\pi'$. The completion with respect to this norm is denoted by $F^*_{\pi}(G)$ and referred to as the \emph{symmetrized $\pi$-pseudofunction algebra} of $G$. It is a Banach $^*$-algebra with the involution coming from $L^1(G)$ (see Proposition 4.2 in \cite{Samei2018QuasiHermitian} where this is proven in a special case; the proof in the general case is analogous with the obvious adjustments). In this paper, we are, in particular, interested in the case where $\pi$ is the left-regular representation $\lambda_p$ of $G$ on $L^p(G)$, for $1<p<\infty$. The symmetrized pseudofunction algebra $F^*_{\lambda_p}(G)$ associated with $\lambda_p$ is referred to as the \emph{symmetrized $p$-pseudofuntion algebra} of $G$.

Let $\mathcal{E}$ be a class of reflexive Banach space and let $\mathcal{E}'$ be the class of Banach spaces which are dual to the spaces in $\mathcal{E}$. We denote by $\Rep_{\mathcal{E}}(G)$ the class of isometric representations of $G$ on spaces in $\mathcal{E}$. We associate to $\mathcal{E}$ a seminorm on $L^1(G)$ as follows: For $f\in L^1(G)$, set
\[
\norm{f}{F^*_{\mathcal{E}}}=\sup\set{\norm{\pi(f)}{}}{\pi\in\Rep_{\mathcal{E}}(G)\mbox{ or }\pi\in\Rep_{\mathcal{E}'}(G)}.
\]
Set $I_{\mathcal{E}}=\bigcap_{\pi\in\Rep_{\mathcal{E}}(G)}\ker\pi$. The seminorm above defines a norm on the quotient of $L^1(G)$ with $I_{\mathcal{E}}\cap I_{\mathcal{E}'}$. We denote by $F^*_{\mathcal{E}}(G)$ the completion with respect to this norm and refer to it as the \emph{symmetrized $\mathcal{E}$-pseudofuntion algebra}. Just like the symmetrized $\pi$-pseudofunction algebra, it is a Banach $^*$-algebra with involution coming from $L^1(G)$. Further, $F^*_{\mathcal{E}}(G)$ is \emph{$\mathcal{E}$-universal} in the sense that, for any $\pi\in\Rep_{\mathcal{E}}(G)$, the identity map on $L^1(G)$ extends to a contraction $F^*_{\mathcal{E}}(G)\rightarrow F^*_{\pi}(G)$.

\paragraph{Spaces of matrix coefficients}
Let $(\pi,E)$ be an isometric representation of the locally compact group $G$ on a Banach space $E$. The \emph{$\pi$-Fourier-Stieltjes space} is the linear subspace of $L^{\infty}(G)$ given by
\[
B_{\pi}(G)=\set{\varphi:G\rightarrow\C\; \mbox{measurable}}{\exists C>0:\abs{\varphi(f)}\leq C\norm{\pi(f)}{}, \forall f\in L^1(G)},
\]
where
\[
\varphi(f)=\integral{f(s)\varphi(s)}{G}{\mu_G(s)}.
\]
We equip $B_{\pi}(G)$ with the norm
\[
\norm{\varphi}{B_{\pi}}=\inf\set{C>0}{\abs{\varphi(f)}\leq C\norm{\pi(f)}{}, \forall f\in L^1(G)}
\]
With this norm, the canonical embedding of $B_{\pi}(G)$ into $L^{\infty}(G)$ is a contraction.

A \emph{matrix coefficient} of the representation $(\pi,E)$ is a function on $G$ of the form
\[
\varphi_{\xi,\eta}(t)=\inner{\pi(t)\xi}{\eta},
\]
for $t\in G$, where $\xi\in E$, $\eta\in E'$ and $\innerdot$ is the duality pairing between $E$ and $E'$. Clearly, $\varphi_{\xi,\eta}$ is an element of $B_{\pi}(G)$ with $\norm{\varphi_{\xi,\eta}}{B_{\pi}}\leq\norm{\xi}{}\norm{\eta}{}$. In general, not all elements of $B_{\pi}(G)$ need to be matrix coefficients of $\pi$, but the $\pi$-Fourier-Stieltjes space can still be be understood very concretely as a space of matrix coefficients. This is made precise in Theorem \ref{thm:Bpi_as_space_of_matrix_coeff} below, which is Theorem 2 in \cite{CowlingFendler1984RepInBanachSp}.
\begin{thm}\label{thm:Bpi_as_space_of_matrix_coeff}
Let $\pi$ be an isometric representation of $G$ on a Banach space $E$. There exists an isometric representation $\pi_0$ on a Banach space $E_0$ such that $B_{\pi}(G)$ and $B_{\pi_0}(G)$ are canonically isometrically isomorphic and such that, for every $\varphi\in B_{\pi}(G)$, one can find $\xi\in E_0$ and $\eta\in E_0'$ such that $\varphi=\inner{\pi_0(\square)\xi}{\eta}$ and $\norm{\varphi}{B_\pi}=\norm{\xi}{}\norm{\eta}{}$. Moreover, if $E$ is an $L^p$-space, or a $QSL^p$-space, then so is $E_0$.
\end{thm}

A $p$-analogue of the Fourier-Stieltjes algebra was proposed by Runde in \cite{Runde2004repQSLP}. Denote by $QSL^p$ the class of all Banach spaces isometrically isomorphic to a quotient of a subspace of an $L^p$-space. Further, we denote by $\Rep_p(G)$ the class of non-degenerate isometric representations of $G$ on a space in $QSL^p$. The \emph{$p$-Fourier-Stieltjes algebra} is the set of matrix coefficients of representations in $\Rep_p(G)$:
\[
B_p(G)=\set{\inner{\pi(\square)\xi}{\eta}}{(\pi,E)\in\Rep_p(G), \xi\in E, \eta\in E'}.
\]
We equip $B_p(G)$ with the following norm: For $\varphi\in B_p(G)$, set
\[
\norm{\varphi}{B_p}=\inf\set{\norm{\xi}{}\norm{\eta}{}}{\varphi=\inner{\pi(\square)\xi}{\eta},\mbox{ for }(\pi,E)\in\Rep_p(G), \xi\in E, \eta\in E'}.
\]
It is shown in \cite{Runde2004repQSLP} that $B_p(G)$ is a commutative Banach algebra over $\C$ with pointwise operations. Clearly, it embeds canonically contractively into $L^{\infty}(G)$.
\begin{warning}
In \cite{Runde2004repQSLP}, Runde defines $B_p(G)$ as the set of matrix coefficients of representations in $\Rep_{p'}(G)$, where $p'$ is the Hölder conjugate of $p$. We follow the convention used in \cite{Daws2006pOperatorSpaces} and do not exchange $p$ and $p'$. Hence, $B_p(G)$ in our notation is $B_{p'}(G)$ in the notation of Runde.
\end{warning}

Let $(\pi,E)$ and $(\rho,F)$ be two isometric Banach space representations of $G$. We say that $\rho$ is \emph{contained} in $\pi$ and write $\rho\leq\pi$ if there exists a linear isometry $T:F\rightarrow E$ such that $\pi(t)T\xi=T\rho(t)\xi$, for all $t\in G$ and $\xi\in F$. If $\pi$ belongs to a class $\mathcal{R}$ of isometric representations of $G$, we say that $\pi$ is \emph{$\mathcal{R}$-universal} if it contains all representations in $\mathcal{R}$. We are, in particular, interested in the class of isometric representations on $QSL^p$-spaces, and we shall say that $\pi$ is \emph{$p$-universal} when it is universal with respect to this class.

\begin{thm}\label{thm:Bpi_and_Bp}
Let $(\pi,E)$ be an isometric $QSL^p$-representation of the locally compact group $G$. Then $B_{\pi}(G)$ embeds canonically and contractively into $B_p(G)$. If $(\pi,E)$ is $p$-universal, this embedding is an isometric isomorphism.
\end{thm}
\begin{proof}
It is a direct consequence of Theorem \ref{thm:Bpi_as_space_of_matrix_coeff} that $B_{\pi}(G)$ embeds contractlively into $B_p(G)$. Indeed, let $\varphi\in B_{\pi}(G)$. By Theorem \ref{thm:Bpi_as_space_of_matrix_coeff} we may find a $QSL^p$-representation $(\pi_0,E_0)$ and elements $\xi\in E$ and $\eta\in E'$ such that $\varphi=\inner{\pi(\square)\xi}{\eta}$ and $\norm{\varphi}{B_{\pi}}=\norm{\xi}{}\norm{\eta}{}$. Then $\varphi$ lies in $B_p(G)$ with $\norm{\varphi}{B_p}\leq\norm{\varphi}{B_{\pi}}$. Suppose now that $(\pi,E)$ is $p$-universal and let $\psi\in B_p(G)$. By definition of $B_p(G)$ we may find a $QSL^p$-representation $(\pi_1,E_1)$ such that $\psi=\inner{\pi_1(\square)\xi}{\eta}$, for some $\xi\in E_1$ and $\eta\in E_1'$. Because $\pi$ is $p$-universal, there is a linear isometry $T:E_1\rightarrow E$ such that $\pi(t)T\zeta=T\pi_1(t)\zeta$, for all $\zeta\in E_1$ and $t\in G$. Then, for any $f\in L^1(G)$ and $\zeta\in E_1$,
\[
\norm{\pi_1(f)\zeta}{}=\norm{T\pi_1(f)\zeta}{}=\norm{\pi(f)T\zeta}{}\leq\norm{\pi(f)}{}\norm{\zeta}{}.
\]
It follows that $\norm{\pi_1(f)}{}\leq \norm{\pi(f)}{}$, and so, $\abs{\psi(f)}=\abs{\inner{\pi_1(f)\xi}{\eta}}\leq \norm{\pi(f)}{}\norm{\xi}{}\norm{\eta}{}$, for every $f\in L^1(G)$. Hence, $\psi$ lies in $B_{\pi}(G)$ with $\norm{\psi}{B_{\pi}}\leq\norm{\xi}{}\norm{\eta}{}$. Since $\psi=\inner{\pi_1(\square)\xi}{\eta}$ was an arbitrary representation of $\psi$ we can take the infimum on the right hand side of this inequality to obtain $\norm{\psi}{B_{\pi}}\leq\norm{\psi}{B_p}$. Hence, when $\pi$ is $p$-universal, the canonical embedding $B_{\pi}(G)\hookrightarrow B_p(G)$ is an isometric isomorphism.
\end{proof}

\paragraph{Complex interpolation}
We give a brief overview of the complex interpolation method focusing on interpolation bounds on families of operators. We refer the reader to \cite{BerghLoefstroem2011Interpolation} for a thorough introduction to the topic. 
A pair $(E_0,E_1)$ of complex Banach spaces is said to be \emph{compatible} if there exists a Hausdorff topological vector space $V$ and $\C$-linear continuous embeddings $\iota_j:E_j\hookrightarrow V$, for $i\in\{0,1\}$. Given a compatible pair $(E_0,E_1)$, their \emph{intersection space} is the vector subspace of $V$ given by $E_0\cap E_1=\iota_0(E_0)\cap\iota_1(E_1)$. It becomes a Banach space when equipped with the norm defined for $\xi=\iota_0(\xi_0)=\iota_1(\xi_1)$ by
\[
\norm{\xi}{E_0\cap E_1}=\max\{\norm{\xi_0}{E_0},\norm{\xi_1}{E_1}\},
\]
Further, the \emph{sum space} of the pair $(E_0,E_1)$ is the vector subspace of $V$ given by $E_0+E_1=\iota_0(E_0)+\iota_1(E_1)$. We equip this space with the Banach space norm
\[
\norm{\xi}{E_0+E_1}=\inf\set{\norm{\xi_0}{E_0}+\norm{\xi_1}{E_1}}{\xi=\iota_0(\xi_0)+\iota_1(\xi_1)}.
\]
The complex interpolation method associates to each parameter $\theta\in[0,1]$ a Banach space $[E_0,E_1]_{\theta}$ such that there are canonical continuous inclusions $E_0\cap E_1\subset [E_0,E_1]_{\theta}\subset E_0+E_1,
$ and such that $E_0\cap E_1$ is dense in $[E_0,E_1]_{\theta}$. For a measure space $(\Omega,\mu)$ and parameters $1<p_0<p_1<\infty$, the interpolation space of $L^{p_0}(\Omega,\mu)$ and $L^{p_1}(\Omega,\mu)$ with parameter $\theta\in[0,1]$ can be identified with $L^{p_{\theta}}(\Omega,\mu)$ with
\[
\frac{1}{p_{\theta}}=\frac{1-\theta}{p_0}+\frac{\theta}{p_1}.
\]
Viewed as a function on $\theta\in[0,1]$, $p_{\theta}$ is continuous and monotonically increasing taking the value $p_0$ at $\theta=0$ and $p_1$ at $\theta=1$. Denote by $S(\Omega)$ and $L^0(\Omega,\mu)$ the simple, respectively, measurable complex valued functions on $\Omega$. Given a linear operator $S(\Omega)\rightarrow L^0(\Omega,\mu)$ which extends to a bounded operator on both $L^{p_0}(\Omega,\mu)$ and $L^{p_1}(\Omega,\mu)$, the Riesz-Thorin Theorem guaranties that it also extends to a bounded operator on the interpolation spaces. The Riesz-Thorin Theorem was generalized by Stein in \cite{Stein1956Interpolation} to families of linear operators. We recall Stein's Interpolation Theorem in Theorem \ref{thm:Stein} below. Denote by
\[
\Strip=\set{z\in\C}{0\leq\real{z}\leq1}=\set{\theta+i\gamma}{0\leq\theta\leq1,\gamma\in\R}
\]
the vertical \emph{strip} in the complex plane. A function $\Phi:\Strip\rightarrow\C$ which is continuous on $\Strip$ and analytic on the interior $\Strip^{\circ}$ is said to have \emph{admissible growth} if there exists a konstant $k<\pi$ such that
\[
\sup_{z\in\Strip}e^{-k\abs{\imag z}}\log\abs{\Phi(z)}<\infty
\]
A family $(T_z)_{z\in\Strip}$ of linear operators $S(\Omega)\rightarrow L^0(\Omega,\mu)$ indexed by the strip is said to be \emph{admissible} if, for every pair of simple functions $f,g\in S(\Omega)$, the map $\Strip\rightarrow\C$ given by
\[
z\mapsto\integral{(T_zf)g}{\Omega}{\mu}
\]
is continuous on $\Strip$, analytic on $\Strip^{\circ}$, and has admissible growth. Given a family $(T_z)_{z\in\Strip}$ of admissible operators such that the operators on the left boundary of the strip extend to bounded operators on $L^{p_0}(\Omega,\mu)$ and the operators on the right boundary of the strip extend to bounded operators on $L^{p_1}(\Omega,\mu)$, Stein's Interpolation Theorem ensures that, for each interpolation parameter $\theta\in[0,1]$, the operator $T_{\theta}$ extends to a bounded operator on the interpolation space $L^{p_{\theta}}(\Omega,\mu)$. 
\begin{thm}[Stein's Interpolation Theorem]\label{thm:Stein}
Let $1\leq p_0<p_1\leq\infty$ and let $(\Omega,\mu)$ be a measure space. Suppose $\net{T}{z}{S}$ is an admissible family of linear operators $S(\Omega)\rightarrow L^0(\Omega,\mu)$ satisfying
\begin{align*}
\norm{T_{i\gamma}\xi}{p_0}\leq M_0(\gamma)\norm{\xi}{p_0}\qquad\mbox{and}\qquad \norm{T_{i\gamma+1}\xi}{p_1}\leq M_1(\gamma)\norm{\xi}{p_1},
\end{align*}
for all simple functions $\xi\in S(\Omega)$ and all $\gamma\in\R$, where $M_j(\gamma)>0$, for $j\in\{0,1\}$, are independent of $\xi$ and satisfy
\begin{align*}
M_j\overdot{=}\sup_{\gamma\in\R}\log M_j(\gamma)<\infty.
\end{align*}
Then, for each $0<\theta<1$, the constant $M_{\theta}>0$ defined by
\begin{align*}
\log M_{\theta}=\frac{\sin\pi\theta}{2}\rintegral{\frac{M_0}{\cosh\pi\gamma-\cos\pi\theta}+\frac{M_1}{\cosh\pi\gamma+\cos\pi\theta}}{-\infty}{\infty}{\gamma}
\end{align*}
is finite, and
\begin{align*}
\norm{T_{\theta}\xi}{p_{\theta}}\leq M_{\theta}\norm{\xi}{p_{\theta}},
\end{align*}
for every simple function $\xi\in S(\Omega)$.
\end{thm} 

\section{Interpolations of $L^1(G)$ and $C^*(G)$}\label{sec:interpolations}
In this section we prove Theorem \ref{introthm:FstarLp->FstarLq} from the introduction, which is Theorem \ref{thm:FstarLp->FstarLq} below. This establishes canonical contractive $^*$-homomorphisms $F_{L^p}^*(G)\rightarrow F_{L^q}^*(G)$, for $1\leq p<q\leq2$. The proof relies on Stein's Interpolation Theorem and on the Banach-Lamperti Theorem. When $q=2$ the Banach-Lamperti Theorem does not apply. However, this obstacle can be circumvented with an application of the so-called ``Gaussian functor trick''.\\

Let $G\act{\sigma}(\Omega,\mu)$ be a measure class preserving action of the locally compact group $G$ on the $\sigma$-finite measure space $(\Omega,\mu)$. Fix $1\leq p_0<p_1<\infty$. For each $s\in G$ and $\gamma\in\R$, define $a^{\gamma}_s:\Omega\rightarrow\C$ by
\begin{align}\label{eqn:def:agamma}
a_s^{\gamma}=\left(\frac{\mathrm{d}s.\mu}{\mathrm{d}\mu}\right)^{\gamma(1/p_0-1/p_1)i}.
\end{align}
\begin{lem}\label{lem:agamma_is_a_1-cocycle}
The map $a^{\gamma}:G\times\Omega\rightarrow\T$ given by $(s,\omega)\mapsto a^{\gamma}_s(\omega)$ with $a_s^{\gamma}$ as in equation (\ref{eqn:def:agamma}) is a $\T$-valued $1$-cocycle for the action $G\act{\sigma}(\Omega,\mu)$.
\end{lem}
\begin{proof}
Recall that the Radon-Nikodym derivative is real-valued and measurable function. Hence, $a^{\gamma}_s\in L^0(\Omega,\mu;\T)$, for all $s\in G$. Further, for each pair $s,t\in G$, the Radon-Nikodym derivative satisfies the following equality:
\[
\frac{\mathrm{d}(st).\mu}{\mathrm{d}\mu}
=\frac{\mathrm{d}s.(t.\mu)}{\mathrm{d}\mu}
=\frac{\mathrm{d}s.(t.\mu)}{\mathrm{d}s.\mu}\frac{\mathrm{d}s.\mu}{\mathrm{d}\mu}.
\]
Taking both sides to the power $\gamma(1/p_0-1/p_1)i$, yields the $1$-cocycle relation:
\[
a^{\gamma}_{st}=(a^{\gamma}_t\circ\sigma_s^{-1})a^{\gamma}_s.\qedhere
\]
\end{proof}
Let $c$ be a $\T$-valued $1$-cocycle for the action $G\act{\sigma}(\Omega,\mu)$. Recall that set of $1$-cocycles is a group with multiplication given entrywise. Hence, $ca^{\gamma}$ is again a $1$-cocycle, for each $\gamma\in\R$, with $a^{\gamma}$ as in Lemma \ref{lem:agamma_is_a_1-cocycle}.

We denote by
\[
\Strip=\set{z\in\C}{0\leq\real{z}\leq1}=\set{\theta+i\gamma}{0\leq\theta\leq1,\gamma\in\R}
\]
the vertical \emph{strip} in the complex plane. To each $z=\theta+i\gamma\in\Strip$ we associate the $1$-cocycle $ca^{\gamma}$ and the Hölder exponent $p_0\leq p_{\theta}\leq p_1$ which is the unique real number such that $1/p_{\theta}=\theta/p_0+(1-\theta)/p_1$. Fix $f\in L^1(G)$. For each $z=\theta+i\gamma\in\Strip$, we define a linear operator $T_{\theta+i\gamma}:S(\Omega)\rightarrow L^0(\Omega,\mu)$ by, for $\xi\in S(\Omega)$ and $\omega\in\Omega$,
\begin{align}\label{eqn:def:Tz}
T_{\theta+i\gamma}\xi(\omega)&=\integral{f(s)(c_sa^{\gamma}_s)(\omega)\left(\frac{\mathrm{d}s.\mu}{\mathrm{d}\mu}\right)^{1/p_{\theta}}(\omega)\xi(s^{-1}.\omega)}{G}{\mu_G(s)}.
\end{align}
That is, $T_{\theta+i\gamma}=\pi_{p_{\theta},\sigma,ca^{\gamma}}(f)$. In particular, we see that $T_{\theta}=\pi_{p_{\theta},\sigma,c}(f)$, $T_{i\gamma}=\pi_{p_0,\sigma,ca^{\gamma}}(f)$ and $T_{1+i\gamma}=\pi_{p_1,\sigma,ca^{\gamma}}(f)$. Lemma \ref{lem:(Tz)_is_admissible} establishes that the family $\net{T}{z}{\Strip}$ is \emph{admissible}.
\begin{lem}\label{lem:(Tz)_is_admissible}
Let $\net{T}{z}{\Strip}$ be the family of bounded linear operators $S(\Omega,\mu)\rightarrow L^0(\Omega,\mu)$ defined in (\ref{eqn:def:Tz}). For every pair of simple functions $\xi,\eta\in S(\Omega,\mu)$, the map $\Strip\rightarrow\C$ given by
\begin{align}\label{eqn:Tz_is_admissible_map:S->C}
z\mapsto\integral{(T_z\xi)\eta}{\Omega}{\mu}
\end{align}
is continuous on $\Strip$ and analytic on $\Strip^{\circ}$. Moreover,
\begin{align}\label{eqn:Tz_is_admissible_sup_log_abs_int}
\sup_{z\in S}\log\abs{\integral{(T_z\xi)\eta}{\Omega}{\mu}}<\infty.
\end{align}
\end{lem}
\begin{proof}
Fix $\xi,\eta\in S(\Omega,\mu)$. The map in (\ref{eqn:Tz_is_admissible_map:S->C}) is continuous on $\Strip$ and analytic on $\Strip^{\circ}$ if and only if the two maps
\begin{align}\label{eqn:Tz_is_admissible_map:S->C_split}
\theta\mapsto\integral{(T_{\theta+i\gamma}\xi)\eta}{\Omega}{\mu}\qquad\mbox{and}\qquad\gamma\mapsto\integral{(T_{\theta+i\gamma}\xi)\eta}{\Omega}{\mu}
\end{align}
are continuous on $[0,1]$ and on $\R$, respectively, and $C^{\infty}$ on $(0,1)$ and on $\R$, respectively.

By the Tonelli-Fubini Theorem, we have
\begin{align*}
\integral{(T_{\theta+i\gamma}\xi)\eta}{\Omega}{\mu}
&=\integral{\eta\integral{f(s)\pi_{p_{\theta},\sigma,ca^{\gamma}}(s)\xi}{G}{\mu_G(s)}}{\Omega}{\mu}\\
&=\integral{f(s)\integral{[\pi_{p_{\theta},\sigma,ca^{\gamma}}(s)\xi]\eta}{\Omega}{\mu}}{G}{\mu_G(s)}
\end{align*}
Consider, for each $s\in G$, the integrand of the inner integral as a function in the three variables $(\theta,\gamma,\omega)\in[0,1]\times\R\times\Omega$:
\begin{align}\label{eqn:Tz_is_admissible_integrand}
(\theta,\gamma,\omega)\mapsto [\pi_{p_{\theta},\sigma,ca^{\gamma}}(s)\xi](\omega)\eta(\omega)=(c_sa^{\gamma}_s)(\omega)\left(\frac{\mathrm{d}s.\mu}{\mathrm{d}\mu}\right)^{1/p_{\theta}}(\omega)\xi(s^{-1}.\omega)\eta(\omega)
\end{align}
This function is measurable in $\omega$ and continuous in $\theta$ and $\gamma$, respectively. Further, since $c_sa^{\gamma}_s$ takes values in $\T$, and since $x^{1/p}$ is monotone as a function in $p$, we have
\begin{align*}
&\abs{[\pi_{p_{\theta},\sigma,ca^{\gamma}}(s)\xi](\omega)\eta(\omega)}\\
&\phantom{c_sa^{\gamma}_s)(\omega)}\leq\max\left\{{\abs{\left(\frac{\mathrm{d}s.\mu}{\mathrm{d}\mu}(\omega)\right)^{1/p_0}\xi(s^{-1}.\omega)\eta(\omega)},\abs{\left(\frac{\mathrm{d}s.\mu}{\mathrm{d}\mu}(\omega)\right)^{1/p_1}\xi(s^{-1}.\omega)\eta(\omega)}}\right\}
\end{align*}
An application of Hölder's inequality and the change of variables formula shows that both terms in this maximum are integrable functions with integrals bounded by $\norm{\xi}{p_0}\norm{\eta}{p'_0}$ and $\norm{\xi}{p_1}\norm{\eta}{p'_1}$, respectively. The Lebesgue Dominated Convergence Theorem then implies continuity in each of the coordinates of the map $[0,1]\times\R\rightarrow\C$ given by
\begin{align}\label{eqn:Tz_is_admissible_fct_in_theta_and_gamma}
(\theta,\gamma)\mapsto\integral{[\pi_{p_{\theta},\sigma,ca^{\gamma}}(s)\xi]\eta}{\Omega}{\mu}.
\end{align}
A second application of the Lebesgue Dominated Convergence Theorem implies continuity of the maps in equation (\ref{eqn:Tz_is_admissible_map:S->C_split}). The argument that the two maps in (\ref{eqn:Tz_is_admissible_map:S->C_split}) are differentiable on $(0,1)$ and on $\R$, respectively, is analogous. Observe that $a^{\gamma}$ is differentiable in $\gamma$ and that $(\mathrm{d}s.\mu/\mathrm{d}\mu)^{1/p_{\theta}}$ is differentiable in $\theta$. We have,
\begin{align*}
\frac{\partial}{\partial\gamma}[\pi_{p_{\theta},\sigma,ca^{\gamma}}\xi](\omega)\eta(\omega)&=\left(\frac{1}{p_0}-\frac{1}{p_1}\right)i\log\left(\frac{\mathrm{d}s.\mu}{\mathrm{d}\mu}(\omega)\right)[\pi_{p_{\theta},\sigma,ca^{\gamma}}\xi](\omega)\eta(\omega),\\
\frac{\partial}{\partial\theta}[\pi_{p_{\theta},\sigma,ca^{\gamma}}\xi](\omega)\eta(\omega)&=\left(\frac{1}{p_0}-\frac{1}{p_1}\right)\log\left(\frac{\mathrm{d}s.\mu}{\mathrm{d}\mu}(\omega)\right)[\pi_{p_{\theta},\sigma,ca^{\gamma}}\xi](\omega)\eta(\omega).
\end{align*}
The function $\log(\mathrm{d}s.\mu/\mathrm{d}\mu)\eta$ lies in $L^p(\Omega,\mu)$, for any $1<p<\infty$, because the measure of the support of $\eta$ is finite. Hence, both partial derivatives are bounded by an integrable function not depending on $\theta$ and $\gamma$. An application of the Lebesgue Dominated Convergence Theorem then implies differentiability of the map of equation (\ref{eqn:Tz_is_admissible_fct_in_theta_and_gamma}) in $\theta$ on $(0,1)$ and in $\gamma$ on $\R$. A second application of the Lebesgue Dominated Convergence Theorem implies differentiability of the maps in equation (\ref{eqn:Tz_is_admissible_map:S->C_split}).

It remains to show the inequality of equation (\ref{eqn:Tz_is_admissible_sup_log_abs_int}). With $p_{\theta}'$ denoting the Hölder conjugate of $p_{\theta}$, Hölder's inequality yields that
\begin{align*}
\abs{\integral{(T_{\theta+i\gamma}\xi)\eta}{\Omega}{\mu}}
\leq\norm{\pi_{q_{\theta},\sigma,ca^{\gamma}}(f)}{}\norm{\xi}{p_{\theta}}\norm{\eta}{p'_{\theta}}
\leq\norm{f}{1}\norm{\xi}{p_{\theta}}\norm{\eta}{p'_{\theta}}.
\end{align*}
The right hand side of this inequality depends only on $\theta$ and not on $\gamma$. Hence,
\[
\sup_{z\in S}\abs{\integral{(T_z\xi)\eta}{\Omega}{\mu}}\leq\norm{f}{1}\sup_{\theta\in[0,1]}\norm{\xi}{p_{\theta}}\norm{\eta}{p'_{\theta}}<\infty.
\]
The inequality of equation (\ref{eqn:Tz_is_admissible_sup_log_abs_int}) follows.
\end{proof}

With an admissible family at hand, we may employ Stein's Interpolation Theorem.
\begin{thm}\label{thm:FstarLp->FstarLq}
Let $G$ be a locally compact second countable group and let $1\leq p<q\leq2$, The identity map on $L^1(G)$ extends to a contractive $^*$-homomorphism $F_{L^p}^*(G)\rightarrow F_{L^q}^*(G)$ with dense range. 
\end{thm}
\begin{proof}
Let $G\act{\sigma}(\Omega,\mu)$ be a measure class preserving action, let $c$ be a $1$-cocycle for this action and consider the representation $\pi_{q,\sigma,c}$ of $G$ on $L^q(\Omega,\mu)$. We denote by $p'$ the Hölder conjugate of $p$. Let $f\in L^1(G)$. For each $z\in\Strip$, let $T_z:S(\Omega,\mu)\rightarrow L^0(\Omega,\mu)$ be the operator defined in (\ref{eqn:def:Tz}) with $p$ in place of $p_0$ and $p'$ in place of $p_1$. That is, for each $\theta\in[0,1]$ and $\gamma\in\R$, $T_{\theta+i\gamma}=\pi_{p_{\theta},\sigma,ca^{\gamma}}(f)$, where $p\leq p_{\theta}\leq p'$ is the unique real number such that $1/p_{\theta}=\theta/p+(1-\theta)/p'$, and where $a^{\gamma}$ be the $1$-cocycle defined, for each $s\in G$, by
\begin{align*}
a^{\gamma}_s=\left(\frac{\mathrm{d}s.\mu}{\mathrm{d}\mu}\right)^{\gamma(1/p-1/p')i}.
\end{align*}
Then $(T_z)_{z\in\Strip}$ is an admissible family by Lemma \ref{lem:(Tz)_is_admissible}. Define $M_j:\R\rightarrow\R_+$, for $j\in\{0,1\}$, by
\begin{align*}
M_0(\gamma)=\norm{\pi_{p,\sigma,ca^{\gamma}}(f)}{}\qquad\mbox{and}\qquad M_1(\gamma)=\norm{\pi_{p',\sigma,ca^{\gamma}}(f)}{}\
\end{align*}
Then $M_j(\gamma)\leq\norm{f}{1}$, for $j\in\{0,1\}$ and for all $\gamma\in\R$. Because $p<q<p'$, we may find a $0<\theta_q<1$ such that $q=p_{\theta_q}$. We obtain from Stein's Interpolation Theorem, and by using that the simple functions are dense in $L^q(\Omega,\mu)$, that $\norm{\pi_{q,\sigma,c}(f)}{q}\leq M_{\theta_q}$, where
\begin{align}\label{eq:Stein_interpolation_bound_for_Mthetap}
\log M_{\theta_q}=\frac{\sin\pi\theta_q}{2}\rintegral{\frac{M_0}{\cosh\pi\gamma-\cos\pi\theta_q}+\frac{M_1}{\cosh\pi\gamma+\cos\pi\theta_q}}{-\infty}{\infty}{\gamma}
\end{align}
and
\begin{align*}
M_0=\sup_{\gamma\in\R}\log \norm{\pi_{p,\sigma,ca^{\gamma}}(f)}{},\qquad\mbox{and}\qquad M_1=\sup_{\gamma\in\R}\log \norm{\pi_{p',\sigma,ca^{\gamma}}(f)}{}.
\end{align*}
The representations appearing in the definitions of $M_0$, respectively, $M_1$, are all representations on $L^p(\Omega,\mu)$, respectively, on $L^{p'}(\Omega,\mu)$. Thus, both $M_0$ and $M_1$ are upper bounded by $\log\norm{f}{F^*_{L^p}}$. We insert this into equation (\ref{eq:Stein_interpolation_bound_for_Mthetap}) to obtain an upper bound on $M_{\theta_q}$:
\begin{align*}
\log M_{\theta_q}&\leq\left(\frac{\sin\pi\theta_q}{2}\rintegral{\frac{1}{\cosh\pi\gamma-\cos\pi\theta_q}+\frac{1}{\cosh\pi\gamma+\cos\pi\theta_q}}{-\infty}{\infty}{\gamma}\right)\log\norm{f}{F^*_{L^p}}.
\end{align*}
A computation shows that the number in the parenthesis on the right hand side of this inequality is always equal to $1$ independently of $\theta_q$:
\[
\frac{\sin\pi\theta_q}{2}\rintegral{\frac{1}{\cosh\pi\gamma-\cos\pi\theta_q}+\frac{1}{\cosh\pi\gamma+\cos\pi\theta_q}}{-\infty}{\infty}{\gamma}
=\left[\frac{1}{\pi}\tan^{-1}\left(\frac{\sinh\pi\gamma}{\sin\pi\theta_q}\right)\right]_{-\infty}^{\infty}
=1.
\]
Hence,
\begin{align}\label{eqn:thm:FstarLp->FstarLq:norm_ineq_obtained_from_Stein}
\norm{\pi_{q,\sigma,c}(f)}{}\leq M_{\theta_q}\leq\norm{f}{F^*_{L^p}}.
\end{align}
The same argument also works with the Hölder conjugate $q'$ in place of $q$.

We have two cases to consider: $q<2$ and $q=2$. Suppose first that $q<2$. Then, by the Banach-Lamperti Theorem, all isometric representations on an $L^q$-space have the form $\pi_{q,\sigma,c}$. Since we have $q'>2$, the isometric representations on an $L^{q'}$-space also have this form. Thus, taking the supremum over all representations of $G$ on any $L^q$- or $L^{q'}$-space, we see that $\norm{f}{F_{L^q}^*}\leq\norm{f}{F^*_{L^p}}$. As $f\in L^1(G)$ was arbitrary, the claim follows.

Suppose now that $q=2$ so that $F^*_{L^q}(G)=C^*(G)$. Let $(\pi,H)$ be a unitary representation of $G$ and let $(\pi_{\R},H_{\R})$ be the orthogonal representation obtained by restriction of scalars. It is clear that $\norm{\pi(f)}{}=\norm{\pi_{\R}(f)}{}$, for any $f\in L^1(G)$. By Corollary A.7.15 in \cite{BekkaDeLaHarpeValette}, we may find a probability space $(\Omega,\mu)$ and a p.m.p. action $G\act{\sigma}(\Omega,\mu)$ such that $\pi_{\R}$ is contained in $\pi_{2,\sigma}\restrict{\R}$. Here, the latter is the representation $\pi_{2,\sigma}$ on $L^2(\Omega,\mu)$ associated with the action $\sigma$ and the trivial $1$-cocycle, but viewed as an orthogonal representation on $L^2(\Omega,\mu;\R)$ via restriction of scalars. Thus, for each $f\in L^1(G)$, we have $\norm{\pi(f)}{}\leq\norm{\pi_{2,\sigma}(f)}{}$. Together with inequality (\ref{eqn:thm:FstarLp->FstarLq:norm_ineq_obtained_from_Stein}), this implies that $\norm{\pi(f)}{}\leq\norm{f}{F^*_{L^p}}$. Since $\pi$ was an arbitrary unitary representation, it follows that $\norm{f}{u}\leq\norm{f}{F^*_{L^p}}$, where $\normdot{u}$ denotes the norm on $C^*(G)$. As $f\in L^1(G)$ was arbitrary, the claim follows also in this case.
\end{proof}

\section{The dual of a symmetrized pseudofunction algebra}\label{sec:dual_of_Fstar}
Let $(\pi,E)$ be an isometric representation of the locally compact group $G$ on a reflexive Banach space $E$. The goal of this section is to describe the dual of the symmetrized $\pi$-pseudofunction algebra, $F^*_{\pi}(G)$. We shall prove Theorem \ref{introthm:Bpi+Bpiprime_is_dual_of_Fstarpi} from the introduction, restated here as Theorem \ref{thm:Bpi+Bpiprime_is_dual_of_Fstarpi}, which identifies $F^*_{\pi}(G)'$ with the sumspace of the $\pi$- and $\pi'$-Fourier-Stieltjes spaces. This is parallel to the non-symmetrized setting where $F_{\pi}(G)'$ can be identified with the $\pi$-Fourier-Stieltjes space -- a fact which we recall in Proposition \ref{prop:Bpi_is_dual_of_Fpi}.\\

Denote by $\innerdot{}$ the duality pairing between $L^1(G)$ and $L^{\infty}(G)$ given, for $f\in L^1(G)$ and $\varphi\in L^{\infty}(G)$, by
\[
\inner{f}{\varphi}=\integral{f(s)\varphi(s)}{G}{\mu_G(s)},
\]
The map $L^{\infty}(G)\rightarrow L^1(G)'$ given by $\varphi\mapsto\inner{\square}{\varphi}$ is a linear isometric isomorphism. Given an isometric representation $(\pi,E)$ of $G$, the canonical linear contraction from $L^1(G)$ to $F_{\pi}(G)$ has dense range by construction. Hence, $F_{\pi}(G)'$ embeds contractively into $L^1(G)'$. We use the identification of $L^1(G)'$ and $L^{\infty}(G)$ in Proposition \ref{prop:Bpi_is_dual_of_Fpi} to identify $F_{\pi}(G)'$ with the $\pi$-Fourier-Stieltjes space, which we recall embeds contractively into $L^{\infty}(G)$.

\begin{prop}\label{prop:Bpi_is_dual_of_Fpi}
Let $(\pi,E)$ be an isometric Banach space representation of the locally compact group $G$. The canonical identification of $L^1(G)'$ with $L^{\infty}(G)$ restricts to an isometric isomorphism between $F_{\pi}(G)'$ and $B_{\pi}(G)$.
\end{prop}
\begin{proof}
Let $\varphi\in B_{\pi}(G)$. By construction of $B_{\pi}(G)$, we have $\abs{\inner{f}{\varphi}}\leq \norm{\varphi}{B_{\pi}}\norm{\pi(f)}{}$, for all $f\in L^1(G)$. Therefore, $\inner{\square}{\varphi}$ extends to a functional $\zeta_{\varphi}$ on $F_{\pi}(G)$ with $\norm{\zeta_{\varphi}}{}\leq \norm{\varphi}{B_{\pi}}$. 
Conversely, for each $\zeta\in F_{\pi}(G)'$ there exists a unique $\varphi_{\zeta}\in L^{\infty}(G)$ such that $\zeta(f)=\inner{f}{\varphi_{\zeta}}$, for all $f\in L^1(G)$. Then $\abs{\inner{f}{\varphi_{\zeta}}}\leq\norm{\zeta}{}\norm{\pi(f)}{}$, for all $f\in L^1(G)$. Hence, $\varphi_{\zeta}$ lies in $B_{\pi}(G)$ and $\norm{\varphi_{\zeta}}{B_{\pi}}\leq\norm{\zeta}{}$.
\end{proof}
\begin{rk}
Together with Theorem \ref{thm:Bpi_and_Bp}, Proposition \ref{prop:Bpi_is_dual_of_Fpi} implies that $F_{\pi}(G)'$ embeds contractively into the $p$-Fourier-Stieltjes algebra $B_p(G)$, and further, that this canonical embedding is an isometric isomorphism when $\pi$ is $p$-universal. This is the content of Theorem 6.6 in \cite{Runde2004repQSLP}.
\end{rk}

Let in the following $(\pi,E)$ be an isometric representation on a reflexive Banach space $E$. We turn our attention to the symmetrized $\pi$-pseudofunction algebra, $F_{\pi}^*(G)$.

\begin{thm}\label{thm:Bpi+Bpiprime_is_dual_of_Fstarpi}
Let $(\pi,E)$ be an isometric representation of the locally compact group $G$ on a reflexive Banach space $E$. The canonical identification of $L^1(G)'$ with $L^{\infty}(G)$ restricts to an isometric isomorphism between $F_{\pi}^*(G)'$ and the sum space $B_{\pi}(G)+B_{\pi'}(G)$.
\end{thm}
\begin{proof}
The direct sum of $\pi$ and its dual representation $\pi'$ defines in a natural way a linear map $\pi\oplus\pi':L^1(G)\rightarrow \mathcal{L}(E)\oplus \mathcal{L}(E')$ given by $(\pi\oplus\pi')(f)=\pi(f)\oplus\pi'(f)$. By construction of $F^*_{\pi}(G)$, this extends to a linear isometry $\pi\oplus\pi':F^*_{\pi}(G)\rightarrow \mathcal{L}(E)\oplus_{\infty} \mathcal{L}(E')$ whose image is a closed subspace. We identify the dual of $\mathcal{L}(E)\oplus_{\infty} \mathcal{L}(E')$ with $\mathcal{L}(E)'\oplus_1 \mathcal{L}(E')'$ and denote by $(\pi\oplus\pi')^*:\mathcal{L}(E)'\oplus_1 \mathcal{L}(E')'\rightarrow F^*_{\pi}(G)'$ the Banach space adjoint of $\pi\oplus\pi'$. Concretely, $(\pi\oplus\pi')^*$ is given by, for $\Phi\in \mathcal{L}(E)'$, $\Psi\in \mathcal{L}(E')'$ and $x\in F^*_{\pi}(G)$,
\begin{align*}
(\pi\oplus\pi')^*(\Phi\oplus \Psi)(x)=\Phi\oplus \Psi((\pi\oplus\pi')(x))=\Phi(\pi(x))+\Psi(\pi'(x)),
\end{align*}
Because $\pi\oplus\pi'$ is an isometric isomorphism onto its image, $(\pi\oplus\pi')^*$ considered as a map $(\imag(\pi\oplus\pi'))'\rightarrow F_{\pi}^*(G)'$ is an isometric isomorphism by Lemma 3.6 in \cite{GardellaThiel2014GpAlgActingOnLp}. Let $\zeta\in F_{\pi}^*(G)'$ and let $\varphi_{\zeta}$ be the unique function in $L^{\infty}(G)$ such that $\zeta(f)=\inner{f}{\varphi_{\zeta}}$, for $f\in L^1(G)$. By the Hahn-Banach Extension Theorem, we can find $\Phi_{\zeta}\in\mathcal{L}(E)'$ and $\Psi_{\zeta}\in\mathcal{L}(E')'$ such that
\[
(\pi\oplus\pi')^*(\Phi_{\zeta}\oplus \Psi_{\zeta})=\zeta\qquad\mbox{and}\qquad \norm{\zeta}{}=\norm{\Phi_{\zeta}}{}+\norm{\Psi_{\zeta}}{}
\]
Then $\Phi_{\zeta}\circ\pi+\Psi_{\zeta}\circ\pi'$, viewn as a function on $G$, satisfies, for all $f\in L^1(G)$,
\[
\zeta(f)=\inner{f}{\Phi_{\zeta}\circ\pi+\Psi_{\zeta}\circ\pi'},
\]
Thus, $\Phi_{\zeta}\circ\pi+\Psi_{\zeta}\circ\pi'$ equals $\varphi_{\zeta}$ by the uniqueness of $\varphi_{\zeta}$. Now $\Phi_{\zeta}\circ\pi$ lies in $B_{\pi}(G)$ with $\norm{\Phi_{\zeta}\circ\pi}{B_{\pi}}\leq\norm{\Phi_{\zeta}}{}$, and $\Psi_{\zeta}\circ\pi'$ lies in $B_{\pi'}(G)$ with $\norm{\Psi_{\zeta}\circ\pi'}{B_{\pi'}}\leq\norm{\Psi_{\zeta}}{}$. Therefore, $\varphi_{\zeta}$ lies in the sum space $B_{\pi}(G)+B_{\pi'}(G)$ with
\[
\norm{\varphi_{\zeta}}{B_{\pi}+B_{\pi'}}\leq\norm{\Phi_{\zeta}}{}+\norm{\Psi_{\zeta}}{}=\norm{\zeta}{}.
\]
Hence, the canonical map $L^1(G)'\rightarrow L^{\infty}(G)$ given by $\zeta\mapsto\varphi_{\zeta}$ restricts to a contraction $F_{\pi}^*(G)'\rightarrow B_{\pi}(G)+B_{\pi'}(G)$. 

Conversely, let $\varphi\in B_{\pi}(G)+B_{\pi'}(G)$. For each $\varepsilon>0$, we can find $\varphi_{1,\varepsilon}\in B_{\pi}(G)$ and $\varphi_{2,\varepsilon}\in B_{\pi'}(G)$ so that
\[
\norm{\varphi_{1,\varepsilon}}{B_{\pi}}+\norm{\varphi_{2,\varepsilon}}{B_{\pi'}}\leq\norm{\varphi}{B_{\pi}+B_{\pi'}}+\varepsilon.
\]
Then, for $f\in L^1(G)$,
\begin{align*}
\abs{\inner{f}{\varphi}}&\leq\abs{\inner{f}{\varphi_{1,\varepsilon}}}+\abs{\inner{f}{\varphi_{2,\varepsilon}}}\\
&\leq\norm{\varphi_{1,\varepsilon}}{B_{\pi}}\norm{\pi(f)}{}+\norm{\varphi_{2,\varepsilon}}{B_{\pi'}}\norm{\pi'(f)}{}\\
&\leq\left(\norm{\varphi}{B_{\pi}+B_{\pi'}}+\varepsilon\right)\norm{f}{F^*_{\pi}}
\end{align*}
Since $\varepsilon>0$ was arbitrary, it follows that $\inner{\square}{\varphi}$ extends to a functional on $F^*_{\pi}(G)$ with $\norm{\inner{\square}{\varphi}}{}\leq\norm{\varphi}{B_{\pi}+B_{\pi'}}$. Hence, the canonical map $L^{\infty}(G)\rightarrow L^1(G)'$ given by $\varphi\mapsto\inner{\square}{\varphi}$ restricts to a contraction $B_{\pi}(G)+B_{\pi'}(G)\rightarrow F^*_{\pi}(G)'$. Because $\varphi\mapsto\inner{\square}{\varphi}$ and $\zeta\mapsto\varphi_{\zeta}$ are inverse of each other, this finishes the proof.
\end{proof}

Since $B_{\pi}(G)$ embeds into $B_p(G)$, for any $QSL^p$-representation $\pi$, we obtain Corollary \ref{cor:Fstarpi_embeds_into_Bp+Bp'} as an immediate corollary to Theorem \ref{thm:Bpi+Bpiprime_is_dual_of_Fstarpi}.

\begin{cor}\label{cor:Fstarpi_embeds_into_Bp+Bp'}
Let $(\pi,E)$ be an isometric $QSL^p$-representation of the locally compact group $G$, for $1<p<\infty$. Then $F^*_{\pi}(G)'$ embeds contractively into the sumspace $B_p(G)+B_{p'}(G)$.
\end{cor}

\section{Applications to amenability}\label{sec:amenability}
In this section, we apply the understanding obtained in Section \ref{sec:dual_of_Fstar} of the Banach space dual of a symmetrized pseudofunction algebra in order to prove Theorem \ref{introthm:amenability_and_Fstarp} from the introduction. This theorem, which we restate as Theorem \ref{thm:amenability_and_Fstarp} below, is parallel to, and builds upon, analogous characterizations of amenability in terms of the $p$-pseudofunction algebra of Herz. We shall give an overview of these classical results as we need them.
\begin{thm}\label{thm:amenability_and_Fstarp}
Let $G$ be a locally compact group and let $1<p,p'<\infty$ be Hölder conjugates. The following are equivalent:
\begin{enumerate}[(i)]
\item $G$ is amenable,
\item $F^*_{\lambda_p}(G)'$ is canonically isometrically isomorphic to the sum space $B_p(G)+B_{p'}(G)$,
\item The canonical map $F^*_{L^p}(G)\rightarrow F^*_{\lambda_p}(G)$ is an isometric isomorphism,
\item The canonical map $F^*_{L^p}(G)\rightarrow F^*_{\lambda_p}(G)$ is an isomorphism,
\item The trivial representation $1_G$ extends to a $^*$-representation of $F^*_{\lambda_p}(G)$.
\end{enumerate}
\end{thm}

The conditions (ii) and (v) in Theorem \ref{thm:amenability_and_Fstarp} describes properties of the Banach space dual of $F^*_{\lambda_p}(G)$. In the non-symmetrized setting, a connection between amenability of the group $G$ and similar properties of the Banach space dual of the $p$-pseudofunction algebra was first established by Cowling in Theorem 5 in \cite{Cowling1979Littlewool-Paley} and later refined by Runde in Theorem 6.7 in \cite{Runde2004repQSLP} and Neufang and Runde in Theorem 4.1 in \cite{NeufangRunde2007OpSpacesOverQSLp}.
\begin{thm}[Cowling \cite{Cowling1979Littlewool-Paley}, Runde \cite{Runde2004repQSLP}, Neufang \& Runde \cite{NeufangRunde2007OpSpacesOverQSLp}]\label{thm:amenability_Fp_Cowling_Runde_Neufang}
Let $G$ be a locally compact group. The following are equivalent:
\begin{enumerate}[(i)]
\item $G$ is amenable,
\item $F_{\lambda_p}(G)'$ is canonically isometrically isomorphic to $B_p(G)$,
\item The trivial representation $1_G$ extends to a representation of $F_{\lambda_p}(G)$.
\end{enumerate}
\end{thm}

We build upon this work in the non-symmetrized setting and apply Theorem \ref{thm:Bpi+Bpiprime_is_dual_of_Fstarpi} to proof the first implication in Theorem \ref{thm:amenability_and_Fstarp}.
\begin{proof}[Proof of Theorem \ref{thm:amenability_and_Fstarp} (i)$\Rightarrow$(ii)]
Assume $G$ is amenable. Then, by Theorem \ref{thm:amenability_Fp_Cowling_Runde_Neufang}, $F_{\lambda_q}(G)'$ is canonically isometrically isomorphic to $B_{q}(G)$, for each $1<q<\infty$. Together with Theorem \ref{thm:Bpi+Bpiprime_is_dual_of_Fstarpi} with $\lambda_p$ in place of $\pi$, this implies that
\[
F^*_{\lambda_p}(G)'\cong F_{\lambda_p}(G)'+F_{\lambda_{p'}}(G)'\cong B_p(G)+B_{p'}(G),
\]
where both isomorphisms are canonical and isometric.
\end{proof}

In $C^*$-algebraic terms, amenability of $G$ is the property that its universal and reduced group $C^*$-algebras coincide canonically. In Theorem 3.7 in \cite{GardellaThiel2014GpAlgActingOnLp}, Gardella and Thiel give an $L^p$-generalization of this. Here, the roles of the universal and reduced $C^*$-algebras are played by the $L^p$-, respectively, $p$-pseudofunction algebras.
\begin{thm}[Gardella \& Thiel \cite{GardellaThiel2014GpAlgActingOnLp}]\label{thm:amenability_Gardella-Thiel}
Let $G$ be a locally compact group and let $1<p<\infty$. The following are equivalent:
\begin{enumerate}[(i)]
\item $G$ is amenable,
\item The canonical map $F_{L^p}(G)\rightarrow F_{\lambda_p}(G)$ is an isometric isomorphism,
\item The canonical map $F_{L^p}(G)\rightarrow F_{\lambda_p}(G)$ is an isomorphism,
\end{enumerate}
\end{thm}
In our Theorem \ref{thm:amenability_and_Fstarp}, the equivalence of amenability and the conditions (iii) and (iv) is parallel to this result of Gardella and Thiel, and our proof of the implications (ii)$\Rightarrow$(iii) and (iv)$\Rightarrow$(v) is an adaptation of theirs. The implication (iii)$\Rightarrow$(iv) in Theorem \ref{thm:amenability_and_Fstarp} is trivial.

\begin{proof}[Proof of Theorem \ref{thm:amenability_and_Fstarp} (ii)$\Rightarrow$(iii)]
By Theorem \ref{thm:Bpi+Bpiprime_is_dual_of_Fstarpi}, the sumspace $B_p(G)+B_{p'}(G)$ can be identified canonically with the dual of $F^*_{QSL^p}(G)$. Now the canonical inclusion of $F^*_{\lambda_p}(G)'$ into the sum space $B_p(G)+B_{p'}(G)$ is the Banach space adjoint of the canonical linear contraction from $F^*_{QSL^p}(G)$ to $F^*_{\lambda_p}(G)$. Thus, it follows from Lemma 3.6 in \cite{GardellaThiel2014GpAlgActingOnLp} that if $F^*_{\lambda_p}(G)'\hookrightarrow B_p(G)+B_{p'}(G)$ is an isometric isomorphism then so is $F^*_{QSL^p}(G)\rightarrow F^*_{\lambda_p}(G)$. Since $F^*_{L^p}(G)$ is intermediate to $F^*_{QSL^p}(G)$ and $F^*_{\lambda_p}(G)$, we see that (iii) follows from (ii).
\end{proof}

\begin{proof}[Proof of Theorem \ref{thm:amenability_and_Fstarp} (iv)$\Rightarrow$(v)]
The trivial representation $1_G$ is an $L^p$-representation. Thus, it extends to a $^*$-representation of $F^*_{L^p}(G)$. Therefore, if the canonical map $F^*_{L^p}(G)\rightarrow F^*_{\lambda_p}(G)$ is an isomorphism, $1_G$ extends to a $^*$-representation of $F^*_{\lambda_p}(G)$.
\end{proof}

Due to work of Hulanicki in \cite{Hulanicki1966MeansFolner} and Reiter in \cite{Reiter1965OnSomeProperties}, amenability of $G$ is characterized as the property that the trivial representation $1_G$ is weakly contained in the left regular representation $\lambda_2$ of $G$ on $L^2(G)$, or, equivalently, that $1_G$ extends to a $^*$-representation of $C^*_r(G)$. The characterization of amenability in Theorem \ref{thm:amenability_Fp_Cowling_Runde_Neufang}(iii) should be understood as an $L^p$-version of the characterization by Hulanicki and Reiter. Indeed, for $p=2$, the $2$-pseudofunction algebra $F_{\lambda_2}(G)$ is nothing but the reduced group $C^*$-algebra of $G$. In the same way, Theorem \ref{thm:amenability_and_Fstarp}(v) is another $L^p$-version of this characterization but this time in the symmetrized setting. The proof of (v)$\Rightarrow$(i) in Theorem \ref{thm:amenability_and_Fstarp} is due to work of Samei \& Wiersma in Proposition 3.1 in \cite{SameiWiersma2018Exotic}. We shall include their proof for the sake of completeness and in order to add to it a few points of clarification. The proof relies on a theorem due to Kesten in \cite{Kesten1959FullBanachMeanVal} (see also Theorem G.4.4 in \cite{BekkaDeLaHarpeValette})

\begin{thm}[Kesten \cite{Kesten1959FullBanachMeanVal}]\label{thm:Kesten}
Let $G$ be a locally compact group and let $\pi$ be a unitary representation of $G$.
\begin{enumerate}[(i)]
\item If $1_G\prec\pi$ then $\norm{\pi(f)}{}=1$, for all $f\in L^1(G)$ with $f\geq0$ and $\norm{f}{1}=1$.
\item If there exists $f\in L^1(G)$ with $f\geq0$, $\norm{f}{1}=1$ and $\norm{\pi(f)}{}=1$, and such that $\supp(f^*\ast f)$ generates a dense subgroup of $G$, then $1_G\prec\pi$.
\end{enumerate}
\end{thm}

When $G$ is $\sigma$-finite, we get the following characterization of amenability as a corollary to Kesten's theorem. We assume this is well known to experts.
\begin{cor}\label{cor:KestenHulanickiReiter}
A locally compact group $G$ is amenable if and only if $\norm{\lambda_2(f)}{}=1$, for all $f\in L^1(G)$ with $f\geq0$ and $\norm{f}{1}=1$.
\end{cor}
\begin{proof}
Suppose $G$ is amenable so that $1_G$ is weakly contained in $\lambda_2$. Then, it follows directly from Theorem \ref{thm:Kesten}(i) that $\norm{\lambda_2(f)}{}=1$, for all $f\in L^1(G)$ with $f\geq0$ and $\norm{f}{1}=1$. For the converse implication, we consider first the case where $G$ is $\sigma$-finite. Assume that $G$ is not amenable so that $1_G$ is not weakly contained in $\lambda_2$. As $G$ is $\sigma$-finite, it decomposes into a countable union of disjoint sets of finite measure. Write $G=\bigcupdot_{n\in\N}F_n$ and define
\[
f=\sum_{n=1}^{\infty}\frac{1}{\mu_G(F_n)2^n}1_{F_n},
\]
where, for each $n\in\N$, $1_{F_n}$ is the indicator function on the set $F_n$. Then $f$ is a non-negative function with $\norm{f}{1}=1$ and with full support, and so, $\supp(f^*\ast f)$ is all of $G$, as well. Then, by Theorem \ref{thm:Kesten}(ii), we must have $\norm{\lambda_2(f)}{}<1$.

Finally, let $G$  be a general (not necessarily $\sigma$-finite) locally compact group. Assume that $\norm{\lambda_2(f)}{}=1$, for all $f\in L^1(G)$ with $f\geq0$ and $\norm{f}{1}=1$. Then, for every open $\sigma$-finite subgroup $H\leq G$ and every $g\in L^1(H)$ with $g\geq0$ and $\norm{g}{1}=1$, we still have $\norm{\lambda_2(g)}{}=1$. Hence, by the above argument, every open $\sigma$-finite subgroup of $G$ is amenable. As $G$ is the union of its open $\sigma$-finite subgroups, it follows that $G$ is amenable.
\end{proof}

Let $f\in L^1(G)$. For parameters $1\leq p_1<p_2<p_3\leq2$, a standard interpolation argument based on the Riesz-Thorin Theorem yields a bound on the norm of $f$ viewed as an element of the symmetrized $p_2$-pseudofunction algebra in terms of the norms of $f$ in the symmetrized $p_1$- and $p_3$-pseudofunction algebras. This was observed and proved By Samei and Wiersma in Proposition 4.5 in \cite{Samei2018QuasiHermitian}. We record this fact without proof in Lemma \ref{lem:Fstarp_interpolating_norms} below.
\begin{lem}\label{lem:Fstarp_interpolating_norms}
Let $G$ be a locally compact group. Let $1\leq p_1<p_2<p_3\leq2$ and let $0<\theta<1$ be such that
\[
\frac{1}{p_2}=\frac{1-\theta}{p_1}+\frac{\theta}{p_3}.
\]
Then, for every $f\in L^1(G)$,
\[
\norm{f}{F^*_{\lambda_{p_2}}}\leq\norm{f}{F^*_{\lambda_{p_1}}}^{1-\theta}\norm{f}{F^*_{\lambda_{p_3}}}^{\theta}.
\]
\end{lem}

We can now give the proof of Samei and Wiersma of the last remaining implication of Theorem \ref{thm:amenability_and_Fstarp}. Observe that the symmetrized $p$-pseudofunction algebra recovers $L^1(G)$ when $p=1$ and $C^*_r(G)$ when $p=2$. Hence, for any parameter $1<p<2$, we can employ Lemma \ref{lem:Fstarp_interpolating_norms} to get an upper bound of the symmetrized $p$-pseudofunction algebra norm in terms of the $L^1$-norm and the reduced $C^*$-norm.

\begin{proof}[Proof of Theorem \ref{thm:amenability_and_Fstarp} (v)$\Rightarrow$(i)]
Assume that $G$ is \emph{not} amenable. By Corollary \ref{cor:KestenHulanickiReiter}, we may then find a function $f\in L^1(G)$ such that $f\geq0$, $\norm{f}{1}=1$ and $\norm{\lambda(f)}{}<1$. As $f\geq0$, we have $1_G(f)=\norm{f}{1}=1$. Further, let $0<\theta<1$ be such that $1/p=(1-\theta)/1+\theta/2$. Lemma \ref{lem:Fstarp_interpolating_norms} with $p_1=1$, $p_2=p$ and $p_3=2$ yields the upper bound
\[
\norm{f}{F^*_{\lambda_p}}\leq \norm{f}{1}^{1-\theta}\norm{\lambda_2(f)}{}^{\theta}<1.
\]
As the involution on $F^*_{\lambda_p}(G)$ is an isometry, it follows from Theorems 11.1.4-5 in \cite{Palmer2001BanachAlgVolII} that any $^*$-representation of $F^*_{\lambda_p}(G)$ must be a contraction. Hence, $1_G$ does not extend to a $^*$-representation of $F^*_{\lambda_p}(G)$.
\end{proof}

\bibliography{bib}
\bibliographystyle{plain}

\vspace{2em}
\begin{minipage}[t]{0.45\linewidth}
  \small    Emilie Mai Elki\ae{}r\\
            Department of Mathematics\\
            University of Oslo\\
            0851 Oslo, Norway\\
		    {\footnotesize elkiaer@math.uio.no}
\end{minipage}

\newpage
\listoffixmes

\end{document}